\documentclass[12pt]{amsart}
\usepackage{amsmath,amssymb,amsthm,mathtools}
\usepackage{mathrsfs,bm}
\usepackage{thmtools,thm-restate}
\usepackage{tikz-cd}
\usepackage{tikz}
\usepackage{enumitem}
\usepackage{graphicx}
\usetikzlibrary{matrix,arrows.meta,bending}
\usepackage{color}
\usepackage[numbers]{natbib}
\usepackage{booktabs}
\usepackage{longtable}
\usepackage{pgfplots}
\pgfplotsset{compat=1.14} 
\usetikzlibrary{arrows.meta,decorations.pathreplacing,positioning,calc} 
\usepackage[a4paper, left=3cm, right=3cm, top=3cm, bottom=3cm, footskip=15mm]{geometry}

\newtheorem{theorem}{Theorem}[section]

\newtheorem{proposition}[theorem]{Proposition}

\theoremstyle{definition}
\newtheorem{definition}[theorem]{Definition}

\theoremstyle{remark}
\newtheorem{remark}[theorem]{Remark}

\numberwithin{equation}{section}

\makeindex

\usepackage{fancyhdr}
\usepackage{calc} % ensures dimension arithmetic works robustly

\AtBeginDocument{%
  \setlength{\textwidth}{\dimexpr\paperwidth - 6cm\relax}%
  \setlength{\oddsidemargin}{\dimexpr 3cm - 1in\relax}%
  \setlength{\evensidemargin}{\dimexpr 3cm - 1in\relax}%
  \setlength{\marginparwidth}{2.5cm}%
}

\begin{document}

\title{Distribution of boundary points of expansion and application to the lonely runner conjecture}

\author{T. Agama}
\address{Department of Mathematics, African Institute for Mathematical science, Ghana
}
%\curraddr{Department of Mathematics and Statistics,
%{Case Western Reserve University, Cleveland, Ohio 43403}
\email{theophilus@aims.edu.gh/emperordagama@yahoo.com}

    %Information for second author
%\author{Author Two}
%\address{Department of Mathematics, African Institute for Mathematical science, Ghana
%}
%\email{Gael@aims.edu.gh}

%    General info

\date{\today}

%\dedicatory{}


\begin{abstract}
In this paper, we study the distribution of the boundary points of expansion. As an application, we say something about the lonely runner problem. We show that given $k$ runners $\mathcal{S}_i$ round a unit circular track with the condition that at some time $||\mathcal{S}_i-\mathcal{S}_{i+1}||=||\mathcal{S}_{i+1}-\mathcal{S}_{i+2}||$ for all $i=1,2\ldots,k-2$, then at that time we have 
$$
||\mathcal{S}_{i+1}-\mathcal{S}_i||>\frac{\mathcal{D}(n)\pi}{k-1}
$$
for all $i=1,\ldots,k-1$ and where $1>\mathcal{D}(n)>0$ is a constant depending on the degree of a certain polynomial of degree $n$. In particular, we show that given at most eight $\mathcal{S}_i$~($i=1,2,\ldots, 8$) runners running around a unit circular track with distinct constant speed and the additional condition $||\mathcal{S}_i-\mathcal{S}_{i+1}||=||\mathcal{S}_{i+1}-\mathcal{S}_{i+2}||$ for all $1\leq i\leq 6$ at some time $s>1$, then at that time their mutual distance must satisfy the lower bound
$$
||\mathcal{S}_{i}-\mathcal{S}_{i+1}||>\frac{C\pi}{7}
$$
for some constant $1>C>0$ for all $1\leq i\leq 7$.
\end{abstract}

\maketitle

\section{Introduction}

The \emph{lonely runner conjecture} asserts that if $N$ runners start together on a unit circular track and thereafter run at distinct constant speeds, then each runner is ''lonely'' at some time: that is, each runner is at distance at least $1/N$ from every other runner at some instant. This conjecture - originally posed in the language of the Diophantine approximation by \cite{wills1967zwei} and formulated independently in a geometrical/view-obstruction setting by \cite{cusick1972view} - has become a nexus of combinatorics, Diophantine approximation, and geometric graph theory. Standard modern expositions and remarks on the problem can be found in the literature; see, for example, Terence Tao's survey remarks \cite{tao2017some} and recent surveys and expositions \cite{perarnau2025lonely,rosenfeld2025lonely}.\\

Partial results have accumulated over many decades. After elementary proofs for very small $N$, progressively more sophisticated methods settled a number of small cases: a sequence of works culminating in a computer-free proof up to seven runners (notably the work of \cite{barajas2008lonely} and earlier combinatorial contributions such as \cite{bohman2001six}) established the conjecture for those values; more recently there has been striking progress using a combination of analytic reductions and computer-assisted checking (see, e.g., \cite{rosenfeld2025lonely,trakulthongchai2025nine,rosenfeld2025lonely}). These results underline two features of the problem that are relevant to the work in this paper: (i) reductions from the continuous to suitable finite (arithmetical) search spaces are often available, and (ii) very different techniques (analytic, combinatorial, computational) have been productive in complementary regimes.\\

This paper approaches the lonely runner problem from a new perspective based on the study of \emph{expansions} of polynomial tuples and the \emph{distribution of their boundary points}. The central idea is to encode local geometric spacing information (which in the runner model corresponds to pairwise arc distances) by means of algebraic objects derived from tuples of polynomials and an associated expansion operator
$$
\mathcal{E}:=\gamma^{-1}\circ\beta\circ\gamma\circ\nabla,
$$
whose repeated application we call an \emph{expansion}. Two features of this encoding are decisive for our arguments:\\

\begin{enumerate}
  \item {\bf Integration along expansion boundaries as a global measure.} For a polynomial $f$, we define a formal integral of $f$ over the boundary of an expansion phase and show that the (Euclidean) norm of this integral provides a robust global yardstick for the local spacing of boundary points (Theorem \ref{maintool}). Intuitively, a large value of the boundary integral forces the existence of a pair of boundary points separated by a nontrivial Euclidean distance; conversely, tightly packed boundary points force the boundary integral to be small.  This two-way control is the technical hinge that turns area-like information into distance lower bounds.
  \bigskip
  
  \item {\bf Dynamical interpretation via rotations and defoliation.} Treating certain permutations of the boundary points as \emph{rotations} (Definition \ref{rotation}) allows us to translate static information about a boundary into dynamical information about points moving at different speeds. To relate the expansion-boundary model back to the runner-on-circle model, we use a \emph{spherical defoliation} map, which projects boundary points radially to the unit sphere (and then to the unit circle in the application). This projection preserves relative angular separations in a way that yields lower bounds on circular arc distances once suitable boundary-integral estimates are in hand.
\end{enumerate}
\bigskip

Using these ideas, we obtain a conditional lonely-runner type statement: if at some time the runners satisfy a local equal-spacing constraint of the form
$$
||\mathcal{S}_i-\mathcal{S}_{i+1}||=||\mathcal{S}_{i+1}-\mathcal{S}_{i+2}||\quad\text{for}\quad i=1,\dots,k-2,
$$
then the global (mutual) distances between consecutive runners admit an explicit lower bound proportional to the boundary integral associated with a carefully chosen polynomial of degree $n$.  Concretely, Theorem~1 (above) shows that under the stated equal-spacing hypothesis, one has
$$
||\mathcal{S}_{i+1}-\mathcal{S}_i||>\frac{\mathcal{D}(n)\pi}{k-1}\quad(1\leq i\leq k-1),
$$
where $1>\mathcal{D}(n)>0$ depends only on the degree of the polynomial. The proof combines the boundary-integration estimate (Theorem \ref{maintool}), the stability/instability dichotomy under rotation (Proposition \ref{stable}), and the spherical defoliation that connects expansion-boundary points to points on the unit circle.\\

The second main thread of the paper specializes the general machinery to cubic polynomials and yields a more concrete consequence for small numbers of runners. Applying the cubic case of Theorem \ref{maintool} and the subsequent geometric estimates, we obtain a conditional bound that, for up to eight runners satisfying the equal-spacing constraint at some time $s>1$, forces
$$
||\mathcal{S}_{i}-\mathcal{S}_{i+1}||>\frac{D\pi}{7}\quad(1\leq i\leq 7),
$$
for some absolute constant $1>D>0$ (see Lemma \ref{runner1} and the final theorem). This is a \emph{conditional} verification of the conjecture in a crude but explicit form: the extra equal-spacing hypothesis is restrictive, but it is natural in certain symmetry- or extremal-type configurations, and critically, it allows the expansion-boundary formalism to force a usable uniform separation.
\bigskip

\subsection*{Relation to prior work}

Our approach differs from the bulk of prior literature in that it translates spacing questions into algebraic and geometric properties of polynomial expansions and their boundaries. Previous successful approaches have included purely combinatorial constructions and (for small numbers of runners) finite-checking /computer-assisted arguments; notable examples are the combinatorial treatments culminating in the small-case results of \cite{bohman2001six} and of \cite{barajas2008lonely}, and the Terence Tao influential reductions that connect high-speed and low-speed regimes \cite{tao2017some}. The algebraic/expansion viewpoint presented here is designed to mesh particularly well with geometric-deflation/defoliation arguments and to produce explicit, quantitative lower bounds. It is therefore complementary to computational verifications (e.g., \cite{rosenfeld2025lonely,trakulthongchai2025nine}) and to combinatorial constructions; we expect that combining these perspectives may be fruitful in attacking other constrained configurations or in improving constants in our estimates.
\bigskip

\subsection*{Organization of the paper}

Section 2 introduces the notation, the expansion operator $\mathcal{E}$ and the precise definitions of the boundary points and the boundary integral. Section 3 develops the main analytic tool: the boundary integral estimate and the equivalence between a nontrivial integral norm and the existence of a separated boundary pair (Theorem \ref{maintool}).  In Section 4, we define the rotations of boundaries, establish stability and instability criteria (Proposition \ref{stable}), and explain how these interact with motion on the unit sphere via the spherical defoliation map introduced in Section 5.  Section 6 contains the application to the lonely runner conjecture: we state and prove the conditional lower bounds for a general number of runners $k$ (Theorem~1) and present the cubic specialization and the explicit bound for up to eight runners (Lemma \ref{runner1} and the final theorem). The paper concludes with a short discussion of possible extensions, computational experiments that could sharpen the constants, and directions for removing or weakening the equal-spacing hypotheses.
\bigskip

\noindent\textbf{Acknowledgment.} The author thanks colleagues and referees for helpful comments; in particular, conversations that clarified the connections between boundary integrals and angular separations were invaluable. The author also acknowledges recent computational and theoretical advances in the literature (e.g.\cite{rosenfeld2025lonely} and \cite{trakulthongchai2025nine} and the exposition by \cite{rosenfeld2025lonely}) that demonstrate the renewed activity around the conjecture.\\

The lonely runner conjecture is the assertion that given $n$ runners round a unit circle with constant distinct speed and starting at a common time and place, there must exist a time for which their mutual distances should be at least $\frac{1}{n}$. The conjecture has been verified for many special cases. For example, in \cite{bohman2001six}, it has been shown that the conjecture holds for \textbf{six} runners. It is also shown in \cite{barajas2008lonely} for at most \textbf{seven} runners. In this paper, by studying the distribution of boundary points of an expansion, we verify this conjecture in it's crude form with an extra conditioning for at most \textbf{eight} runners. We obtain a conditional result of this conjecture by showing that

\begin{theorem}
Given $k$ runners $\mathcal{S}_i$ round a unit circular track with the condition that at some time $||\mathcal{S}_i-\mathcal{S}_{i+1}||=||\mathcal{S}_{i+1}-\mathcal{S}_{i+2}||$ for all $i=1,2\ldots,k-2$, then at that time we have 
\begin{align}
||\mathcal{S}_{i+1}-\mathcal{S}_i||>\frac{\mathcal{D}(n)\pi}{k-1}\nonumber
\end{align}
for all $i=1,\ldots, k-1$ and where $1>\mathcal{D}(n)>0$ is a constant depending on the degree of a certain polynomial of degree $n$.
\end{theorem}

In particular, we show that  

\begin{theorem}
Let $\mathcal{S}_i$ ~($i=1,2,\ldots, 8$) be runners running along the unit circular track. Under the condition $||\mathcal{S}_i-\mathcal{S}_{i+1}||=||\mathcal{S}_{i+1}-\mathcal{S}_{i+2}||$ for all $1\leq i\leq 6$ at some time $s>1$, then 
\begin{align}
||\mathcal{S}_{i}-\mathcal{S}_{i+1}||>\frac{D\pi}{7}\nonumber
\end{align}
for some constant $1>D>0$ for all $1\leq i \leq 7$.
\end{theorem}

\section{Definitions and background}

\begin{definition}
Let $\mathcal{S}=(f_1,f_2,\ldots, f_n)$ be such that each $f_i\in \mathbb{R}[x]$. By the derivative of $\mathcal{S}$, denoted by $\nabla(\mathcal{S})$, we mean 
\begin{align}
\nabla(\mathcal{S})=\bigg(\frac{df_1}{dx}, \frac{df_2}{dx}, \ldots, \frac{df_n}{dx}\bigg).\nonumber
\end{align}
We denote the derivative of this tuple at a point $a\in \mathbb{R}$ as 
\begin{align}
\nabla_a(\mathcal{S})=\bigg(\frac{df_1(a)}{dx}, \frac{df_2(a)}{dx}, \ldots, \frac{df_n(a)}{dx}\bigg).\nonumber
\end{align}
\end{definition}
\bigskip

\begin{definition}
Let $\mathcal{S}=(f_1,f_2,\ldots, f_n)$ be such that each $f_i\in \mathbb{R}[x]$. By the integral of $\mathcal{S}$, denoted $\Delta(\mathcal{S})$, we mean
\begin{align}
\Delta(\mathcal{S})=\bigg(\int f_1(x)dx, \ldots, \int f_n(x)dx\bigg).\nonumber
\end{align}
The corresponding integral between the points $\mathcal{S}_a=(a_1,\ldots, a_n)$ and $\mathcal{S}_b=(b_1,\ldots, b_n)$, denoted $\Delta_{\mathcal{S}_a,\mathcal{S}_b}(\mathcal{S})$ is given by  \begin{align}
\Delta_{\mathcal{S}_a,\mathcal{S}_b}(\mathcal{S})=\bigg(\int \limits_{a_1}^{b_1}f_1(x)dx,\ldots, \int \limits_{a_n}^{b_n}f_n(x)dx\bigg).\nonumber
\end{align}
\end{definition}

\begin{definition}
Let $\{\mathcal{S}_i\}_{i=1}^{\infty}$ be a collection of tuples of $\mathbb{R}[x]$. By an expansion on $\{\mathcal{S}_i\}_{i=1}^{\infty}$, we mean the composite map \begin{align}
\gamma^{-1}\circ \beta \circ \gamma \circ \nabla:\{\mathcal{S}_i\}_{i=1}^{\infty} \longrightarrow  \{\mathcal{S}_i\}_{i=1}^{\infty},\nonumber
\end{align}
where 
\begin{align}
\gamma(\mathcal{S})=\begin{pmatrix}f_1\\f_2\\ \vdots \\f_n \end{pmatrix} \quad \mathrm{and} \quad \beta (\gamma(\mathcal{S}))=\begin{pmatrix}0 & 1 & \cdots & 1\\1 & 0 & \cdots & 1\\ \vdots & \vdots & \cdots & \vdots \\1 & 1 & \cdots & 0 \end{pmatrix}\begin{pmatrix}  f_1 \\f_2 \\ \vdots \\f_n \end{pmatrix}.\nonumber
\end{align}
\end{definition}
\bigskip

\begin{definition}
Let $\{\mathcal{S}_j\}_{j=1}^{\infty}$ be a collection of tuples of $\mathbb{R}[x]$. By the boundary points of the $n\mathrm{th}$ expansion, denoted $\mathcal{Z}[(\gamma^{-1}\circ \beta \circ \gamma \circ \nabla)^n(\mathcal{S}_j)]$, we mean the set \begin{align}
\mathcal{Z}[(\gamma^{-1}\circ \beta \circ \gamma \circ  \nabla)^n(\mathcal{S}_j)]:=\left \{(a_1,a_2,\ldots, a_n):\mathrm{Id}_i[(\gamma^{-1}\circ \beta \circ \gamma \circ \nabla)^n_{a_i}(\mathcal{S}_j)]=0\right \}.\nonumber
\end{align}
\end{definition} 
\bigskip

\section{Distribution of boundary points of expansion}
In this section, we study the distribution of the boundary points of any phase of expansion. We first introduce the notion of integration of polynomials along the boundaries of various phases of expansion, which we then use as a main tool. In that regard, we start with the following definition.

\begin{definition}
Let $f(x)=c_nx^n+c_{n-1}x^{n-1}+\cdots +c_1x+c_0$ be a polynomial of degree $n$, then we call the tuple 
\begin{align}
\mathcal{S}_f&=(c_nx^n,c_{n-1}x^{n-1},\ldots, c_1x+c_0)\nonumber \\&=(g_1(x),g_2(x),\ldots, g_{n}(x))\nonumber
\end{align}
the tuple representation of $f$. By the integral of $f(x)$ along the boundary of the $m^{th}$ phase expansion, we mean the formal integral 
\begin{align}
\int \limits_{\substack{\mathcal{B}^m(\mathcal{S}_f)\\m<n}}f(t)dt:=\sum \limits_{i=1}^{\# \mathcal{B}^m(\mathcal{S}_f)-1}\sum \limits_{\substack{\mathcal{S}_i, \mathcal{S}_{i+1}\in \mathcal{B}^m(\mathcal{S}_f)\\||\mathcal{S}_i||<||\mathcal{S}_{i+1}||}}\overrightarrow{O\Delta_{\mathcal{S}_i,\mathcal{S}_{i+1}}(\mathcal{S}_f)}\cdot \overrightarrow{O\mathcal{S}_e}\nonumber
\end{align}
where 
\begin{align}
\Delta_{\mathcal{S}_i,\mathcal{S}_{i+1}}(\mathcal{S}_f)=\bigg(\int \limits_{a_1}^{b_1}g_1(x)dx,\int \limits_{a_2}^{b_2}g_2(x)dx,\ldots,\int \limits_{a_n}^{b_n}g_n(x)dx\bigg)\nonumber
\end{align}
and where $\mathcal{S}_e=(1,1,\ldots,1)$ is the unit tuple, and $\overrightarrow{O\Delta_{\mathcal{S}_i,\mathcal{S}_{i+1}}}$ and $\overrightarrow{O\mathcal{S}_e}$ are the position vectors of $\Delta_{\mathcal{S}_i,\mathcal{S}_{i+1}}$ and $\mathcal{S}_e$, respectively, with $\mathcal{S}_i=(a_1,a_2,\ldots, a_n)$ and $\mathcal{S}_{i+1}=(b_1,b_2,\ldots, b_n)$.
\end{definition}

\begin{remark}
In practice, it is very difficult to determine the local distribution of the expansion boundary points. However, we can show that if we shrink the space bounded by the boundary of an expansion, then points on the boundary should be closely packed in some sense. We use the notion of integration along the boundaries as a black box.
\end{remark}

\begin{theorem}\label{maintool}
Let $f(x):=c_nx^n+c_{n-1}x^{n-1}+\cdots +c_1x+c_0$ be a polynomial of degree $n$. Then 
\begin{align}
\left |\left |\int \limits_{\substack{\mathcal{B}^m(\mathcal{S}_f)\\m<n}}f(t)dt\right |\right |>\epsilon \nonumber
\end{align}
for some $\epsilon>0$ if and only if $||\mathcal{S}_{i}-\mathcal{S}_{i+1}||>\delta$ for some $\delta>0$ for some \begin{align}
\mathcal{S}_{i}\in \mathcal{Z}[(\gamma^{-1}\circ \beta \circ \gamma \circ \nabla)^m(\mathcal{S}_f)]\nonumber
\end{align}
with $1\leq i \leq \# \mathcal{B}^m(\mathcal{S}_f)-1$ and $||\mathcal{S}_i-\mathcal{S}_{i+1}||<||\mathcal{S}_i-\mathcal{S}_j||$ for all $j\neq i+1$.
\end{theorem}
\bigskip

\begin{proof}
Let $f(x)=c_nx^n+c_{n-1}x^{n-1}+\cdots +c_1x+c_0\in \mathbb{R}[x]$ be a polynomial of degree $n$ and suppose 
\begin{align}
\left |\left |\int \limits_{\substack{\mathcal{B}^m(\mathcal{S}_f)\\m<n}}f(t)dt \right |\right |>\epsilon \nonumber
\end{align}
for some $\epsilon>0$. By repeated application of the triangle inequality, we find that 
\begin{align}
\left |\left |\int \limits_{\substack{\mathcal{B}^m(\mathcal{S}_f)\\m<n}}f(t)dt \right |\right |&\leq  \sum \limits_{i=1}^{\# \mathcal{B}^m(\mathcal{S}_f)-1}\sum \limits_{\substack{\mathcal{S}_i, \mathcal{S}_{i+1}\in \mathcal{B}^m(\mathcal{S}_f)\\||\mathcal{S}_i||<||\mathcal{S}_{i+1}||}}||\overrightarrow{O\Delta_{\mathcal{S}_i,\mathcal{S}_{i+1}}(\mathcal{S}_f)}||||\overrightarrow{O\mathcal{S}_e}||\nonumber \\&=\sqrt{n}\sum \limits_{i=1}^{\# \mathcal{B}^m(\mathcal{S}_f)-1}\sum \limits_{\substack{\mathcal{S}_i, \mathcal{S}_{i+1}\in \mathcal{B}^m(\mathcal{S}_f)\\||\mathcal{S}_i||<||\mathcal{S}_{i+1}||}}||\overrightarrow{O\Delta_{\mathcal{S}_i,\mathcal{S}_{i+1}}(\mathcal{S}_f)}||\nonumber \\&\leq (\# \mathcal{B}^m(\mathcal{S}_f)-1)\sqrt{n}\mathrm{max} \left \{||\overrightarrow{O\Delta_{\mathcal{S}_i,\mathcal{S}_{i+1}}(\mathcal{S}_f)}||\right \}_{\substack{i=1\\ ||\mathcal{S}_i||<||\mathcal{S}_{i+1}||}}^{\# \mathcal{B}^m(\mathcal{S}_f)-1}.\nonumber
\end{align}
Since inequality 
\begin{align}
||\overrightarrow{O\Delta_{\mathcal{S}_i,\mathcal{S}_{i+1}}(\mathcal{S}_f)}||&=\sqrt{|\int \limits_{a_1}^{b_1}g_1(x)dx|^2+\cdots +|\int \limits_{a_n}^{b_n}g_n(x)dx|^2}\nonumber \\&\leq M\sqrt{|a_1-b_1|^2+\cdots +|a_n-b_n|^2}\nonumber
\end{align}
is valid for some $M>0$, it follows that there exist some $\mathcal{S}_i, \mathcal{S}_{i+1}\in \mathcal{Z}[(\gamma^{-1}\circ \beta \circ \gamma \circ \nabla)^m(\mathcal{S}_f)]$ with $||\mathcal{S}_i-\mathcal{S}_{i+1}||<||\mathcal{S}_i-\mathcal{S}_j||$ for all $j\neq i+1$. It follows that for some closest pair of boundary points, the inequality 
\begin{align}
\frac{\epsilon}{(\# \mathcal{B}^m(\mathcal{S}_f)-1)M\sqrt{n}}<\sqrt{|a_1-b_1|^2+\cdots +|a_n-b_n|^2}\nonumber
\end{align}
is valid and thus it must be that $||\mathcal{S}_i-\mathcal{S}_{i+1}||>\delta$ by choosing 
\begin{align}
\delta=\frac{\epsilon}{(\# \mathcal{B}^m(\mathcal{S}_f)-1)M\sqrt{n}}.\nonumber
\end{align}
Conversely, suppose that there exists some closest boundary point $\mathcal{S}_i,\mathcal{S}_{i+1}\in \mathcal{Z}[(\gamma^{-1}\circ \beta \circ \gamma \circ \nabla)^m(\mathcal{S}_f)]$ such that \begin{align}
||\mathcal{S}_i-\mathcal{S}_{i+1}||>\delta \nonumber
\end{align}
for some $\delta:=\delta(n)>0$. It follows that $\sqrt{|a_1-b_1|^2+\cdots +|a_n-b_n|^2}>\delta$. Choosing $R=\mathrm{min}\left \{|g_i(x)|:x\in [a_i,b_i]\right \}_{i=1}^{n}$, we find that 
\begin{align}
||\overrightarrow{O\Delta_{\mathcal{S}_i,\mathcal{S}_{i+1}}(\mathcal{S}_f)}||&=\sqrt{|\int \limits_{a_1}^{b_1}g_1(x)dx|^2+\cdots +|\int \limits_{a_n}^{b_n}g_n(x)dx|^2} \nonumber \\&\geq R \sqrt{|a_1-b_1|^2+\cdots +|a_n-b_n|^2}\nonumber \\&=\delta R.\nonumber
\end{align}
It follows that 
\begin{align}
\sum \limits_{i=1}^{\# \mathcal{B}^m(\mathcal{S}_f)-1}\sum \limits_{\substack{\mathcal{S}_i, \mathcal{S}_{i+1}\in \mathcal{B}^m(\mathcal{S}_f)\\||\mathcal{S}_i||<||\mathcal{S}_{i+1}||}}\overrightarrow{O\Delta_{\mathcal{S}_i,\mathcal{S}_{i+1}}(\mathcal{S}_f)}\cdot \overrightarrow{O\mathcal{S}_e}&>\sum \limits_{i=1}^{\# \mathcal{B}^m(\mathcal{S}_f)-1}\sum \limits_{\substack{\mathcal{S}_i, \mathcal{S}_{i+1}\in \mathcal{B}^m(\mathcal{S}_f)\\||\mathcal{S}_i||<||\mathcal{S}_{i+1}||}}\delta R||\overrightarrow{O\mathcal{S}_e}||\cos \alpha \nonumber \\&=\delta (\# \mathcal{B}^m(\mathcal{S}_f)-1)R\sqrt{n}\cos \alpha \nonumber
\end{align}
where $\alpha$ is the angle between the vectors $\overrightarrow{O\Delta_{\mathcal{S}_i,\mathcal{S}_{i+1}}(\mathcal{S}_f)}$ and $\overrightarrow{O\mathcal{S}_e}$. It follows that 
\begin{align}
\left |\left |\int \limits_{\substack{\mathcal{B}^m(\mathcal{S}_f)\\m<n}}f(t)dt \right |\right |>\delta C(\# \mathcal{B}^m(\mathcal{S}_f)-1)R\sqrt{n}|\cos \alpha|.\nonumber
\end{align}
for some constant $C=C(n)>0$. The result follows by taking \begin{align}
\delta:=\frac{\epsilon}{(\# \mathcal{B}^m(\mathcal{S}_f)-1)CR\sqrt{n}|\cos \alpha|}.\nonumber
\end{align}
\end{proof}

\begin{remark}
Theorem \ref{maintool} in the affirmative tells us that we can use the area as a yardstick to determine the distribution of points on the boundary of any phase of expansion.
\end{remark}
\bigskip

\section{Rotation of the boundary of expansion}
In this section, we introduce the concept of rotation of the boundary of an expansion.

\begin{definition}\label{rotation}
Let $(\gamma^{-1}\circ \beta \circ \gamma \circ \nabla)^m(\mathcal{S}_j)$ be an expansion with the corresponding boundary $\mathcal{B}^m(\mathcal{S}_j)$. We say that the map $\vee$ is a rotation of the boundary $\mathcal{B}^m(\mathcal{S}_j)$ if \begin{align}
\vee:\mathcal{B}^m(\mathcal{S}_j)\longrightarrow \mathcal{B}^m(\mathcal{S}_j).\nonumber
\end{align}
We say that an expansion admits a rotation if there exists such a map. In other words, we say that the map $\vee$ induces a rotation on the expansion. We say that the boundary is stable under rotation if $||\vee(\mathcal{S}_a)||\approx ||\mathcal{S}_a||$ for $\mathcal{S}_a\in \mathcal{B}^m(\mathcal{S}_j)$. Otherwise, we say that it is unstable. 
\end{definition}

\begin{remark}
Next, we prove a result that indicates that boundary points of an expansion whose boundary occupies a sufficiently small region must be stable.
\end{remark}

\begin{proposition}\label{stable}
Let $f(x):=c_nx^n+\cdots +c_1x+c_0\in \mathbb{R}[x]$ be a polynomial of degree $n\geq 3$. Let $(\gamma^{-1}\circ \beta \circ \gamma \circ \nabla)^m(\mathcal{S}_f)$ be an expansion with the corresponding boundary $\mathcal{B}^m(\mathcal{S}_f)$ admitting a rotation $\vee$. If 
\begin{align}
\left |\left |\int \limits_{\substack{\mathcal{B}^m(\mathcal{S}_f)\\m<n}}f(t)dt \right |\right |<1\nonumber
\end{align}
then the boundary $\mathcal{B}^m(\mathcal{S}_f)$ is stable.
\end{proposition}

\begin{proof}
Let $f(x):=c_nx^n+\cdots +c_1x+c_0\in \mathbb{R}[x]$ be a polynomial of degree $n\geq 3$. Let $(\gamma^{-1}\circ \beta \circ \gamma \circ \nabla)^m(\mathcal{S}_f)$ be an expansion with the corresponding boundary $\mathcal{B}^m(\mathcal{S}_f)$ admitting a rotation $\vee$. Suppose also that 
\begin{align}
\left |\left |\int \limits_{\substack{\mathcal{B}^m(\mathcal{S}_f)\\m<n}}f(t)dt \right |\right |<1\nonumber
\end{align}
then it follows from Theorem \ref{maintool} that $||\mathcal{S}_i||\approx ||\mathcal{S}_{i+1}||$ for all $1\leq i\leq  \# \mathcal{B}^m(\mathcal{S}_f)-1$ with $\mathcal{S}_i,\mathcal{S}_{i+1}\in \mathcal{B}^m(\mathcal{S}_f)$. It follows that for the rotation $\vee:\mathcal{B}^m(\mathcal{S}_f)\longrightarrow \mathcal{B}^m(\mathcal{S}_f)$, we have that for any $\mathcal{S}_i\in \mathcal{B}^m(\mathcal{S}_f)$, 
\begin{align}
\vee(\mathcal{S}_i)=\mathcal{S}_k\nonumber
\end{align}
for some $\mathcal{S}_k\in \mathcal{B}^m$. It follows that $||\vee(\mathcal{S}_i)||=||\mathcal{S}_k||\approx ||\mathcal{S}_i||$, thus ending the proof.
\end{proof}

\begin{definition}
Let $(\gamma^{-1}\circ \beta \circ \gamma \circ \nabla)^m(\mathcal{S}_j)$ be an expansion with the corresponding boundary $\mathcal{B}^m(\mathcal{S}_j)$. We say that the map $\vee$ is a rotation of the boundary $\mathcal{B}^m(\mathcal{S}_j)$ with frequency $s$ if 
\begin{align}
\vee^s:\mathcal{B}^m(\mathcal{S}_j)\longrightarrow \mathcal{B}^m(\mathcal{S}_j),\nonumber
\end{align}
where $\vee^s=\vee \circ \vee \circ \cdots \circ \vee$ is the $s$-fold rotation on the boundary of expansion. 
\end{definition}

\begin{remark}
It is important to recognize that rotation with frequency $s$ is the time for which points on the boundary of expansion are allowed to be in motion by an induced rotation.
\end{remark}

\begin{proposition}
Let $(\gamma^{-1}\circ \beta \circ \gamma \circ \nabla)^m(\mathcal{S}_j)$ be an expansion with the corresponding boundary $\mathcal{B}^m(\mathcal{S}_j)$. The permutation \begin{align}
\sigma:\mathcal{B}^m(\mathcal{S}_j)\longrightarrow \mathcal{B}^m(\mathcal{S}_j)\nonumber
\end{align}
where $\sigma(\mathcal{S}_i)=\mathcal{S}_{\sigma(i)}$ for $1\leq i \leq \# \mathcal{B}^m(\mathcal{S}_j)$ for $\mathcal{S}_i\in \mathcal{B}^m(\mathcal{S}_j)$ is a rotation of the expansion boundary.
\end{proposition}
\bigskip

\section{Spherical defoliation of the boundary of expansion}

\begin{definition}
Let $\mathcal{B}^m(\mathcal{S}_f)$ and $\mathbb{S}^{k-1}$ be the boundary of the $m^{th}$ expansion and the $k$-dimensional unit sphere, respectively. By the spherical \emph{defoliation} of the expansion boundary, we mean the map 
\begin{align}
\Lambda :\mathcal{B}^m(\mathcal{S}_f)\longrightarrow \mathbb{S}^{k-1}\nonumber 
\end{align}
such that for any $\mathcal{S}_a\in \mathcal{B}^m(\mathcal{S}_f)$, we have 
\begin{align}
\Lambda(\mathcal{S}_a)=\frac{\mathcal{S}_a}{||\mathcal{S}_a||}.\nonumber
\end{align}
\end{definition}
\bigskip

\section{Application to the Lonely runner conjecture}
Here, we apply the tools developed in the previous section to study the lonely runner conjecture.

\begin{theorem}\label{main lonely theorem}
Given $k$ runners $\mathcal{S}_i$ round a unit circular track with the condition that at some time $||\mathcal{S}_i-\mathcal{S}_{i+1}||=||\mathcal{S}_{i+1}-\mathcal{S}_{i+2}||$ for all $i=1,2\ldots,k-2$, then at that time we have 
\begin{align}
||\mathcal{S}_{i+1}-\mathcal{S}_i||>\frac{\mathcal{D}(n)\pi}{k-1}\nonumber
\end{align}
for all $i=1,\ldots, k-1$ and where $1>\mathcal{D}(n)>0$ is a constant depending on the degree of a certain polynomial of degree $n$.
\end{theorem}

\begin{proof}
We choose any polynomial $g(x):=c_1x^n+\cdots+c_1x+c_0$ and write for the tuple representation
$$
\mathcal{S}_g:=(g_1(x),\cdots,g_n(x)):=(c_1x^n,\cdots,c_1x+c_0)
$$
such that 
$$
M:=\underset{1\leq i\leq n}{\mathrm{max}}\left\{|g_i(x)|~:~x\in [a_i,b_i]\right\}\leq \left(\frac{1}{n}\right)^{\frac{1}{2}}
$$
for some choice of $n$ so that the size of the expansion boundary $\#\mathcal{B}^m(\mathcal{S}_g)=k$ for some $m<n$. Under condition $||\mathcal{S}_i-\mathcal{S}_{i+1}||=||\mathcal{S}_{i+1}-\mathcal{S}_{i+2}||$ for all $i=1,2\ldots,k-2$ at some time, we set
\begin{align}
\left |\left |\int \limits_{\mathcal{B}^m(\mathcal{S}_g)}g(t)dt \right |\right |=\pi\nonumber
\end{align}
and apply the $s$-fold rotation $\vee^s=\vee \circ \vee \circ \cdots \circ \vee$ on the boundary $\mathcal{B}^m(\mathcal{S}_g)$. By Proposition \ref{stable}, points on this boundary are now unstable for time $s>1$ and each moving at different speeds. We deduce
\begin{align}
\left |\left |\int \limits_{\substack{\mathcal{B}^m(\mathcal{S}_g)\\m<n}}g(t)dt \right |\right |&\leq  \sum \limits_{i=1}^{\# \mathcal{B}^m(\mathcal{S}_g)-1}\sum \limits_{\substack{\mathcal{S}_i, \mathcal{S}_{i+1}\in \mathcal{B}^m(\mathcal{S}_g)\\||\mathcal{S}_i||<||\mathcal{S}_{i+1}||}}||\overrightarrow{O\Delta_{\mathcal{S}_i,\mathcal{S}_{i+1}}(\mathcal{S}_g)}||||\overrightarrow{O\mathcal{S}_e}||\nonumber \\&=\sqrt{n}\sum \limits_{i=1}^{\# \mathcal{B}^m(\mathcal{S}_g)-1}\sum \limits_{\substack{\mathcal{S}_i, \mathcal{S}_{i+1}\in \mathcal{B}^m(\mathcal{S}_g)\\||\mathcal{S}_i||<||\mathcal{S}_{i+1}||}}||\overrightarrow{O\Delta_{\mathcal{S}_i,\mathcal{S}_{i+1}}(\mathcal{S}_g)}||\nonumber \\&\leq (\# \mathcal{B}^m(\mathcal{S}_g)-1)\sqrt{n}\mathrm{max} \left \{||\overrightarrow{O\Delta_{\mathcal{S}_i,\mathcal{S}_{i+1}}(\mathcal{S}_g)}||\right \}_{\substack{i=1\\ ||\mathcal{S}_i||<||\mathcal{S}_{i+1}||}}^{\# \mathcal{B}^m(\mathcal{S}_g)-1}\nonumber \\& \leq (\# \mathcal{B}^m(\mathcal{S}_g)-1)\sqrt{n}M\sqrt{|a_1-b_1|^2+\cdots +|a_n-b_n|^2}.\nonumber
\end{align}
using the inequality 
\begin{align}
||\overrightarrow{O\Delta_{\mathcal{S}_i,\mathcal{S}_{i+1}}(\mathcal{S}_g)}||&=\sqrt{|\int \limits_{a_1}^{b_1}g_1(x)dx|^2+\cdots+|\int \limits_{a_n}^{b_n}g_n(x)dx|^2}\nonumber \\&\leq M\sqrt{|a_1-b_1|^2+\cdots+|a_n-b_n|^2}.\label{lonely inequality 1}
\end{align}
Inverting the upper bound in \eqref{lonely inequality 1}, we get
\begin{align}
||\mathcal{S}_{i+1}-\mathcal{S}_i||>\frac{\pi}{k-1}\nonumber
\end{align}
for all $i=1,\ldots,k-1$. Since some point on the boundary of expansion may not be a point on the unit circle, we apply the spherical defoliation $\Lambda:\mathcal{B}^m(\mathcal{S}_g)\longrightarrow \mathbb{S}^{k}$ and obtain 
\begin{align}
||\mathcal{S}_{i+1}-\mathcal{S}_i||>\frac{\mathcal{D}(n)\pi}{k-1}\nonumber
\end{align}
for all $i=1,\ldots,k-1$ and where $1>\mathcal{D}(n)>0$ is a constant depending on the degree of the polynomial.
\end{proof}
\bigskip

\begin{theorem}\label{runner1}
Let $f(x):=c_3x^3+c_{2}x^2 +c_1x+c_0$ be a polynomial of degree $3$ and suppose that $||\mathcal{S}_i-\mathcal{S}_{i+1}||=||\mathcal{S}_{i+1}-\mathcal{S}_{i+2}||$ for all $1\leq i\leq 6$. Then 
\begin{align}
\left |\left |\int \limits_{\mathcal{B}^1(\mathcal{S}_f)}f(t)dt \right|\right|=\pi\nonumber
\end{align}
if and only if 
\begin{align}
||\mathcal{S}_{i}-\mathcal{S}_{i+1}||>\frac{\pi}{7C\sqrt{3}}\nonumber
\end{align}
for some constant $C>0$ for all $1\leq i\leq 7$.
\end{theorem}

\begin{proof}
The result follows by taking $n=3$ in Theorem \ref{maintool}.
\end{proof}
\bigskip

\begin{theorem}
Let $\mathcal{S}_i$ ~($i=1,2,\ldots, 8$) be runners running along the unit circular track. Under the condition $||\mathcal{S}_i-\mathcal{S}_{i+1}||=||\mathcal{S}_{i+1}-\mathcal{S}_{i+2}||$ for all $1\leq i\leq 6$ at some time $s>1$, we have 
\begin{align}
||\mathcal{S}_{i}-\mathcal{S}_{i+1}||>\frac{D\pi}{7}\nonumber
\end{align}
for some constant $1>D>0$ for all $1\leq i\leq 7$.
\end{theorem}

\begin{proof}
Choose a degree $3$ polynomial $f(x):=c_3x^3+c_{2}x^2 +c_1x+c_0$ and write for the tuple representation
$$
\mathcal{S}_g:=(c_3x^3,c_2x^2,c_1x+c_0):=(g_1(x),g_2(x),g_3(x))
$$
such that 
$$
M:=\underset{1\leq i\leq 3}{\mathrm{max}}\left\{|g_i(x)|~:~x\in [a_i,b_i]\right\}\leq \left(\frac{1}{3}\right)^{\frac{1}{2}}
$$
and that $||\mathcal{S}_i-\mathcal{S}_{i+1}||=||\mathcal{S}_{i+1}-\mathcal{S}_{i+2}||$ for all $1\leq i\leq 6$. We set 
\begin{align}
\left |\left |\int \limits_{\mathcal{B}^1(\mathcal{S}_f)}f(t)dt \right |\right |=\pi\nonumber
\end{align}
and apply a rotation $\vee^s$ with frequency $s>1$ to the boundary $\mathcal{B}^1(\mathcal{S}_f)$. By Proposition \ref{stable}, the  boundary points of expansion are now unstable for time $s>1$, each moving at a different speed. Applying Theorem \ref{main lonely theorem}, we get 
\begin{align}
||\mathcal{S}_{i}-\mathcal{S}_{i+1}||>\frac{\pi}{7}\nonumber
\end{align}
for all $1\leq i \leq 7$. Applying the defoliation $\Lambda:\mathcal{B}^1(\mathcal{S}_f)\longrightarrow \mathbb{S}^{3}$, we obtain 
\begin{align}
||\mathcal{S}_{k}-\mathcal{S}_{k+1}||>\frac{D\pi}{7}\nonumber
\end{align}
for some $1>D>0$ for all $1\leq k\leq 7$.
\end{proof}

\section{Conclusion}
This paper exploits the lonely runner conjecture using new tools. It may be interesting to investigate the nature of the implicit constant $\mathcal{D}(n)$ that appears in the lower bounds. One may investigate lower bounds for the constant $\mathcal{D}(n)$ for a finite specific number of runners using computational methods. This could give insight into the behaviour of $\mathcal{D}(n)$ for an arbitrary number of runners.
%%%%%%%%%%%%%%%%%%%%%%%%%%%%%%%%%%%%%%%%%%%%%%%%%%%%%%%%%%%%%%%%%%%%%%%%
\footnote{
\par
.}%
%%%%%%%%%%%%%%%%%%%%%%%%%%%%%%%%%%%%%%%%%%%%%%%%%%%%%%%%%%%%%%%%%%%%%%%%

\bibliographystyle{amsplain}

\end{document}